\documentclass[a4paper,11pt]{article}

\usepackage{amssymb}
\usepackage{amsthm}
\usepackage{tikz}
\usepackage{tikz,pgfplots}
\usepackage{tkz-graph}
\usepackage{tkz-berge}
\usetikzlibrary{decorations.markings}
\usepackage{amsmath,amssymb}
\usepackage{todonotes}
\usepackage[T1]{fontenc}
\pgfplotsset{compat=1.18}
\usepackage{cite}
\usepackage{color}
\usepackage{hyperref}
\usepackage[ruled,vlined]{algorithm2e}
\usepackage{algpseudocode}

\tikzstyle{black vertex}=[scale=0.6, fill=black, draw=black, shape=circle]
\tikzstyle{white vertex}=[scale=0.6, fill=white, draw=black, shape=circle]
\tikzstyle{blue vertex}=[scale=0.7, fill=blue, draw=blue, shape=circle]
\tikzstyle{red vertex}=[scale=0.7, fill=red, draw=red, shape=circle]
\tikzstyle{none}=[fill=none, draw=none, shape=circle]

\hypersetup{colorlinks=true}

\hypersetup{colorlinks=true, linkcolor=blue, citecolor=blue,urlcolor=blue}

\usepackage{graphicx}
\usepackage{placeins,caption}

\newcommand{\chimu}{\chi_{\mu}}

\DeclareMathOperator{\cp}{\,\square\,}
\newcommand{\diam}{{\rm diam}}
\newcommand{\rad}{{\rm rad}}

\usepackage[top=3cm,bottom=3cm,left=2.6cm,right=2.6cm]{geometry}
	\newtheorem{theorem}{Theorem}[section]
	\newtheorem{lemma}[theorem]{Lemma}
	\newtheorem{corollary}[theorem]{Corollary}
	\newtheorem{proposition}[theorem]{Proposition}

	\newtheorem{problem}[theorem]{Problem}
	\newtheorem{question}[theorem]{Question}
	
	\theoremstyle{definition}

        \newtheorem{claim-proof}{Claim}

\begin{document}
	
\title{Coloring the vertices of a graph with \\ mutual-visibility property}

\author{
Sandi Klav\v{z}ar $^{a,b,c}$ \\
\texttt{\footnotesize sandi.klavzar@fmf.uni-lj.si}  \\
{\footnotesize ORCID: 0000-0002-1556-4744}
\and
Dorota Kuziak $^{d}$ \\
\texttt{\footnotesize dorota.kuziak@uca.es} \\
{\footnotesize ORCID: 0000-0001-9660-3284}
\and
Juan Carlos Valenzuela Tripodoro $^{e}$ \\
\texttt{\footnotesize jcarlos.valenzuela@uca.es} \\
{\footnotesize ORCID: 0000-0002-6830-492X}
\and
Ismael G. Yero $^{e}$ \\
\texttt{\footnotesize ismael.gonzalez@uca.es} \\
{\footnotesize ORCID: 0000-0002-1619-1572}
}
\maketitle

\begin{center}
$^a$ Faculty of Mathematics and Physics, University of Ljubljana, Slovenia\\

$^b$ Institute of Mathematics, Physics and Mechanics, Ljubljana, Slovenia\\

$^c$ Faculty of Natural Sciences and Mathematics, University of Maribor, Slovenia\\

$^d$ Departamento de Estad\'istica e IO, Universidad de C\'adiz, Algeciras Campus, Spain\\

$^e$ Departamento de Matem\'aticas, Universidad de C\'adiz, Algeciras Campus, Spain
\end{center}

\maketitle
	
\begin{abstract}
Given a graph $G$, a mutual-visibility coloring of $G$ is introduced as follows. We color two vertices $x,y\in V(G)$ with a same color, if there is a shortest $x,y$-path whose internal vertices have different colors than $x,y$. The smallest number of colors needed in a mutual-visibility coloring of $G$ is the mutual-visibility chromatic number of $G$, which is denoted $\chimu(G)$. Relationships between $\chimu(G)$ and its two parent ones, the chromatic number and the  mutual-visibility number, are presented. Graphs of diameter two are considered, and in particular the asymptotic growth of the mutual-visibility number of the Cartesian product of complete graphs is determined. A greedy algorithm that finds a mutual-visibility coloring is designed and several possible scenarios on its efficiency are discussed. Several bounds are given in terms of other graph parameters such as the diameter, the order, the maximum degree, the degree of regularity of regular graphs, and/or the mutual-visibility number. For the corona products it is proved that the value of its mutual-visibility chromatic number depends on that of the first factor of the product. Graphs $G$ for which $\chimu(G)=2$ are also considered.
\end{abstract}


\noindent
AMS Subj.\ Class.\ (2020): 05C12, 05C15, 05C69

\section{Introduction}

Given a connected graph $G = (V(G),E(G))$ and a set of vertices $S\subseteq V(G)$, it is said that two vertices $x,y\in S$ are $S$-\textit{visible} if there is a shortest $x,y$-path $P$ such that
$V(P)\cap S=\{x,y\}$. The set $S$ is called a \textit{mutual-visibility set} of $G$ if any two vertices of $S$
are $S$-visible. The cardinality of a largest mutual-visibility set of $G$ is the \textit{mutual-visibility number} of $G$, denoted by $\mu(G)$. By a {\em $\mu$-set} of $G$ we mean a mutual-visibility set of $G$ with cardinality $\mu(G)$.

The concepts described above were introduced and studied for the first time by Di Stefano in 2022 in~\cite{DiStefano} motivated in part by some model of robot navigation in networks avoiding collisions between themselves. Soon after, the variety of  mutual-visibility problems in graphs consisting of (total, dual, outer) mutual-visibility sets was presented in~\cite{variety-2023}. In just a few years, the area has blossomed, not only because of the original computer science motivation on robot visibility, but perhaps above all for the reason, that the visibility problems are intrinsically connected with several classical combinatorial problems, as for instance the Zarankiewicz problem, see~\cite{Cicerone-2023}, and Tur\'an type problems, see~\cite{boruzanli-2024, Bujtas, Cicerone-2024b}. Additional interesting contributions to the problem are~\cite{Axenovich-2024, Bresar, Cicerone-hyper, Cicerone-hered, Cicerone-2024a, Korze-2024a, Korze-2024b, kuziak-2023, tian-2023+}.

When we look at the mutual-visibility problem from a practical point of view, we might need more than just one mutual-visibility set, instead we wish to partition the vertex set of a graph into mutual-visibility sets. Therefore, we present here a new research direction on the mutual-visibility problem as follows. Given the connected graph $G$, we color its
vertices by using the following rule. Two vertices $x,y\in V(G)$ are colored with a same color, if there exists a shortest $x,y$-path
whose internal vertices have different colors than $x,y$. Clearly, if $x,y$ are adjacent, then they can be colored equal. Such coloring
shall be called a \textit{mutual-visibility coloring} of $G$. The smallest number of colors needed in a mutual-visibility coloring
of $G$ is the \textit{mutual-visibility chromatic number} of $G$, and will be denoted by $\chimu(G)$. In order to justify the terminology,
notice that in a mutual-visibility coloring, each color class forms a mutual-visibility set.

The rest of the article is structured as follows. Section \ref{sec:prelim} contains some preliminary first results, including a lower bound of $\chimu(G)$ as a function of the mutual-visibility number of the graph, as well as, a relationship with the classical chromatic number for graphs of diameter two. In Section \ref{sec:diameter2}, we study the case of graphs with diameter two. Particularly, the asymptotic growth of the mutual-visibility chromatic number of the Cartesian product of complete graphs is established. A greedy algorithm will be described in Section \ref{sec:greedy}. It will be proved there that although the
algorithm is optimal in many cases, it can also return arbitrarily bad results as a consequence of an erroneous choice of the mutual-visibility sets at each stage of the algorithm. Section \ref{sec:bounds} describes
several bounds as a function of various structural parameters such as the diameter, the order, the maximum degree, the degree of regularity of regular graphs, and/or the mutual-visibility number. A realizability result is also proved, showing that for every feasible value, there is a graph that attains it as a mutual-visibility chromatic number. This section also contains a Nordhauss-Gaddum type upper bound for the mutual-visibility chromatic number. Section \ref{sec:corona} explores the corona product of two graphs $G$ and $H$, by showing that its mutual-visibility chromatic number is either the mutual-visibility chromatic number of the first factor of the product, or that one plus one. Section \ref{sec:chimu2} is focused on finding all those graphs $G$ for which $\chimu(G)=2$, since this is the smallest possible value that such parameter can achieve, and which means that $V(G)$ can be partitioned into two mutual-visibility sets. Finally, we close this work with Section \ref{sec:conclud}, where we outline some remarks and open problems that might be of interest to continue this investigation.

\section{Preliminaries}
\label{sec:prelim}

Along our whole exposition, for a positive integer $k$ we shall write $[k]=\{1,\dots, k\}$. Let $G$ be a graph and let $G_i$, $i\in [k]$, be its connected components. Then it is evident that
$$\chimu(G) = \sum_{i=1}^k \chimu(G_i)\,.$$
For this reason, from now on all the graphs considered are connected unless stated otherwise. Also, all our graphs have neither loops nor multiple edges.

A first basic connection between $\chimu(G)$ and $\mu(G)$ is as follows. Let $\mathcal{P}=\{ P_i:\ i\in \left[\chimu(G)\right] \}$ be a partition of $V(G)$ into mutual-visibility sets. Clearly, $|P_i|\le \mu(G)$ for all $i \in \left[\chimu(G)\right]$. This already implies the following lower bound which we state as a lemma for subsequent use.

\begin{lemma}
    \label{lem:simple-lower-bound}
If $G$ is a graph, then $$\chimu(G) \ge \left\lceil\frac{n(G)}{\mu(G)}\right\rceil\,.$$
\end{lemma}

On the other hand, the following arguments justify the term chromatic in our definition and relate our investigation with the classical one of vertex coloring. A \textit{proper coloring} of the vertex set of a graph $G$ is an assignment of labels (or colors) to the vertices of $G$ in such a way that each two adjacent vertices have different labels (colors). The \textit{chromatic number} of $G$, denoted $\chi(G)$, represents the smallest number of colors among all possible proper colorings of $G$. This parameter is one of the classical ones in graph theory, and there are lots of variations of it. For more information on this fact, see, for instance, the book~\cite{jensen-1995}.

It can be observed that a proper coloring of an arbitrary graph $G$ does not induce a mutual-visibility coloring in general. However, if we consider graphs of diameter two, then such relationship becomes true. That is, let $G$ be a connected graph of diameter two. Since a set of vertices $S$ having a same color in any proper coloring $G$ forms an independent set (it induces an edgeless graph), we hence deduce that such $S$ must be a mutual-visibility set of $G$. Thus, the partition of $V(G)$ induced by the proper coloring, represents also a partition into mutual-visibility sets for $G$. This leads to the following result.

\begin{proposition}
\label{pro:chimu-chi}
If $G$ is a graph of diameter two, then $\chimu(G)\le \chi(G)$.
Moreover, this bound is tight.
\end{proposition}

\begin{proof}
The bound clearly follows from the previous comments. To see that it is tight, consider the complete bipartite graphs $K_{r,t}$ with $(r,t)\ne (1,1)$, for which $\chimu(K_{r,t})=2= \chi(K_{r,t})$.
\end{proof}

We close this section with some other extra terminology and notation that shall be further used.

Let $G$ be a graph. The \textit{maximum degree} of $G$ is denoted by $\Delta(G)$, and its order by $n(G)$. Given a vertex $v\in V(G)$, the \textit{open neighborhood} of $v$ is denoted by $N_G(v)$. If $X\subseteq V(G)$ and $F\subseteq E(G)$, then the subgraphs induced by $X$ and by $F$ are respectively denoted by $G[X]$ and $G[F]$.  The \textit{complement} of $G$ is denoted by $\overline{G}$.

The \textit{distance} between vertices $u,v\in V(G)$ is denoted by $d_G(u,v)$. The \textit{diameter} $\diam(G)$ of $G$ is the maximum distance bwteen its vertices.
The graph $G$ is {\em geodetic} if each pair of its vertices is connected by a unique shortest path. A subgraph $H$ of $G$ is a {\em geodetic subgraph} if it has the same property, that is, each pair of vertices of $H$ is connected in $G$ by a unique shortest path. A subgraph $H$ of $G$ is a {\em convex subgraph} if for any vertices $x$ and $y$ of $H$, every shortest $x,y$-path in $G$ lies completely in $H$.

The {\em Cartesian product} $G\cp H$ and the {\em strong product} $G\boxtimes H$ of graphs $G$ and $H$ both have the vertex set $V(G)\times V(H)$. The vertices $(g, h$) and $(g', h')$ are adjacent in $G\cp H$ if either $gg'\in E(G)$ and $h = h'$, or $g = g'$ and $hh'\in E(H)$. By a {\em layer} of $G\cp H$ we mean a subgraph induced by all the vertices in which one coordinate is fixed. Note that a layer is either isomorphic to $G$ or to $H$. The vertices $(g, h$) and $(g', h')$ are adjacent in $G\boxtimes H$ if either one of the two condition for the Cartesian product holds, or $gg'\in E(G)$ and $hh'\in E(H)$.

\section{Diameter two graphs}
\label{sec:diameter2}

In order to continue the flow initiated in Section \ref{sec:prelim} with graphs of diameter two, we next consider this class of graphs with a bit more detail. In the main result we determine the asymptotic growth of $\chimu(K_n\cp K_n)$, before that we prove the following.

\begin{theorem}
\label{thm:diam-2-C4}
If $G$ is a graph with $\diam(G) = 2$ and each pair of vertices at distance two lies in a $4$-cycle, then $\chimu(G) = \mathcal{O}(\sqrt{\Delta(G)})$.
\end{theorem}

\begin{proof}
Consider an arbitrary partition ${\cal E} = \{E_1, \ldots, E_\ell\}$ of $E(G)$ such that no part of it contains $C_4$ as a subgraph. Based on ${\cal E}$, we form a partition ${\mathcal V} = \{V_1, \ldots, V_\ell\}$ of $V(G)$ as follows. Let $v\in V(G)$ and let $v$ lies in $G[E_i]$, $i\in I$. Then we put $v$ into $V_k$, where $k = \min_{i\in I} i$. Note that some of the sets $V_i$ might be empty, and let $j$ be the largest index such that $V_j\ne \emptyset$. Assume ${\mathcal V}' = \{V_1, \ldots, V_j\}$. Then ${\mathcal V}'$ is a partition of $V(G)$. Note also that since $G[E_i]$, $i\in [j]$, contains no $C_4$, the same holds for $G[V_i]$.

We claim that ${\mathcal V}'$ forms a mutual-visibility coloring. For this sake consider a nonempty part $V_i\in {\mathcal V}$  and let $x,y\in V_i$. If $xy\in E(G)$, then they are $V_i$-visible. Assume next that $d_G(x,y) = 2$. Then by our assumption, $x$ and $y$ lie in a $C_4$. But now at least one vertex of this $C_4$ does not belong to $V_i$, which in turn implies that $x$ and $y$ are again $V_i$-visible. Hence $V_i$ is a mutual-visibility set for each $i\in [j]$ and thus ${\mathcal V}'$ forms a mutual-visibility coloring.

To complete our argument, we recall from~\cite[Theorem~1]{kang-2015} that every graph $G$ admits a decomposition of $E(G)$ into $\mathcal{O}(\sqrt{\Delta(G)})$ parts, such that none of them contains $C_4$ as a subgraph.
\end{proof}

The next result implies that the upper bound of Theorem~\ref{thm:diam-2-C4} is asymptotically tight.

\begin{theorem}
\label{thm:Hamming}
If $n$ is large enough, then $\chimu(K_n\cp K_n) = \Theta (\sqrt{n})$.
\end{theorem}

\begin{proof}
First note that $K_n\cp K_n$ fulfills the assumptions of Theorem~\ref{thm:diam-2-C4}, hence $\chimu(K_n\cp K_n) = \mathcal{O}(\sqrt{n})$.

To prove that $\chimu(K_n\cp K_n) = \Omega(\sqrt{n})$, we recall from~\cite[Corollary~3.7]{Cicerone-2023} that if $m, n\ge 2$, then $\mu(K_m\cp K_n) = z(m,n;2,2)$, where $z(m, n; 2, 2)$ is the maximum number of $1$s that an $m\times n$ binary matrix can have, provided that it contains no $2\times 2$ submatrix of $1$s. (To determine the value $z(m, n; 2, 2)$ is an instance of the Zarankiewicz's problem, see~\cite{west-2021}.) When $n$ is sufficiently large, the value $z(n, n; 2, 2)$ can be bounded as follows~\cite{brown-1966, erdos-1966}:
$$\mu(K_n\cp K_n)  \le z(n,n;2,2) \le \frac{1}{2} n(1+ \sqrt{4n-3})\,.$$
This implies that, provided that $n$ is large enough, a largest mutual-visibility set of $K_n\cp K_n$ is of order $n^{3/2}$. Therefore, we need $\Theta(\sqrt{n})$ mutual-visibility sets to partition $V(K_n\cp K_n)$, that is, $\chimu(K_n\cp K_n) = \Omega(\sqrt{n})$.
\end{proof}

Theorem~\ref{thm:diam-2-C4} also shows that the bound given in Proposition \ref{pro:chimu-chi} is in general not achieved, since it is well known from \cite{Sabidussi} that $\chi(K_n\cp K_n)=n$.

\section{Greedy mutual-visibility coloring algorithm}
\label{sec:greedy}

A natural greedy algorithm for coloring a graph with mutual-visibility sets is to select at each step a largest mutual-visibility set of $G$ among the vertices which are not yet colored, and color the selected set with a new color. This is formalized in Algorithm~\ref{alg:greedy}.

\begin{algorithm}
\caption{Greedy mutual-visibility coloring}
\begin{algorithmic}[1]
\State\textbf{Input:}~{Connected graph $G$}.
\State\textbf{Output:}~{Mutual-visibility coloring of $G$}.
\State $V = V(G)$, $c = 1$
\State while $V\ne \emptyset$:
\State \quad determine a largest mutual-visibility set $X$ of $G$, where $X\subseteq V$
\State \quad color the vertices of $X$ with $c$
\State \quad $V = V \setminus X$, $c = c+1$
 \end{algorithmic}
 \label{alg:greedy}
\end{algorithm}

Algorithm~\ref{alg:greedy} is optimal in many cases. It will follow from our subsequent results, that it works optimally on block graphs (provided at each step of the algorithm the set of simplicial vertices is selected, which is indeed a largest mutual-visibility set at each step). For another example consider the strong grids $H_k = P_{2k}\boxtimes P_{2k}$, $k\ge 1$. If follows from~\cite[Theorem~4.4]{Cicerone-2024a} that $\mu(H_k) = 8k-4$. Moreover, the set of all vertices of $H_k$ which are not of maximum degree (that is, the set of boundary vertices of $H_k$) forms a $\mu$-set of $H_k$. (It can be actually proved that this set is the unique $\mu$-set of $H_k$.) Then Algorithm~\ref{alg:greedy} will color these vertices with color 1, and proceeding by induction, the algorithm will color $H_k$ with $k$ colors. On the other hand, the main diagonal of $H_k$ is the unique shortest path between its end vertices. Because this path contains $2k$ vertices, we infer that $\chimu(H_k)\ge k$ which in turn implies that Algorithm~\ref{alg:greedy} returns an optimal coloring.

The next result provides another example which demonstrates that Algorithm~\ref{alg:greedy} is optimal, and at the same time it gives a family of graphs for which the bound of Lemma~\ref{lem:simple-lower-bound} is sharp.

\begin{proposition}
If $t = 6k$, $k\ge 3$, then $\chimu(C_t \cp C_t) = 2k$.
\end{proposition}

\begin{proof}
By~\cite[Proposition~3.3]{Korze-2024a} we have $\mu(C_t\cp C_t) = 3t$. Hence by Lemma~\ref{lem:simple-lower-bound} we get $\chimu(C_t\cp C_t) \ge \lceil t^2/3t \rceil = t/3 = 2k$.

In the proof of \cite[Proposition~3.3]{Korze-2024a}, a set $M$ of cardinality $3t$ is constructed and proved to be a mutual-visibility set of $C_t\cp C_t$. We will not repeat here the explicit (slightly complicated) definition of this set, but instead identify its key properties. In each layer with respect to the first factor of $C_t\cp C_t$, the set $M$ has exactly three vertices which are uniformly spaced at distance $2k$, that is, the three vertices from $M$ which lie in a given layer are pairwise at distance $2k$. By the transitivity of $C_t\cp C_t$, $X$ can be respectively shifted $2k-1$ times, each time by $1$ in the first coordinate, to construct sets $M_2, M_3, \ldots, M_{2k}$. Using the transitivity of $C_t\cp C_t$ again, each of the sets $M_2, M_3, \ldots, M_{2k}$ is a mutual-visibility set. Since $V(C_t\cp C_t) = M\bigcup_{i=2}^{2k}M_i$, we have thus found a mutual-visibility coloring of $C_t\cp C_t$ using $2k$ colors. We conclude that $\chimu(C_t\cp C_t) \le 2k$ and we are done.
\end{proof}

We next show that Algorithm~\ref{alg:greedy} is not optimal in general. For $k\ge 1$, let $G_k$ be the graph obtained from $C_4$ by amalgamating $k$ private $4$-cycles to each of the four edges of the original $C_4$, see Fig.~\ref{fig:Gk}. Setting $V(C_4) = [4]$, we denote the remaining vertices as can be seen from the figure. Note that $n(G_k) = 8k + 4$.

\begin{figure}[ht!]
\begin{center}
\begin{tikzpicture}[scale=0.85,style=thick]

		\node [style=white vertex] (0) at (0, 3) {};
		\node [style=black vertex] (1) at (3, 3) {};
		\node [style=white vertex] (2) at (0, 6) {};
		\node [style=black vertex] (3) at (3, 6) {};
		\node [style=black vertex] (4) at (1, 7) {};
		\node [style=white vertex] (5) at (2, 7) {};
		\node [style=black vertex] (6) at (1, 9) {};
		\node [style=white vertex] (7) at (2, 9) {};
		\node [style=white vertex] (8) at (4, 5) {};
		\node [style=white vertex] (9) at (4, 4) {};
		\node [style=white vertex] (10) at (6, 5) {};
		\node [style=white vertex] (11) at (6, 4) {};
		\node [style=white vertex] (12) at (2, 2) {};
		\node [style=black vertex] (13) at (1, 2) {};
		\node [style=white vertex] (14) at (2, 0) {};
		\node [style=black vertex] (15) at (1, 0) {};
		\node [style=black vertex] (16) at (-1, 4) {};
		\node [style=black vertex] (17) at (-1, 5) {};
		\node [style=black vertex] (18) at (-3, 5) {};
		\node [style=black vertex] (19) at (-3, 4) {};
		\node [style=black vertex] (20) at (1, 7) {};

  		\node [style=none] (21) at (4.6, 4) {$(12)_1$};
		\node [style=none] (22) at (2.5, 3.5) {$1$};
		\node [style=none] (23) at (2.5, 5.5) {$2$};
		\node [style=none] (24) at (0.5, 5.5) {$3$};
		\node [style=none] (25) at (0.5, 3.5) {$4$};
		\node [style=none] (27) at (6.6, 5) {$(21)_k$};
		\node [style=none] (28) at (6.6, 4) {$(12)_k$};
		\node [style=none] (29) at (4.6, 5) {$(21)_1$};
		\node [style=none] (30) at (0.95, 7.5) {$(32)_1$};
		\node [style=none] (32) at (2.05, 7.5) {$(23)_1$};
		\node [style=none] (33) at (0.7, 9.4) {$(32)_k$};
		\node [style=none] (34) at (2.3, 9.4) {$(23)_k$};
		\node [style=none] (35) at (1.0, 1.5) {$(41)_1$};
		\node [style=none] (36) at (2.0, 1.5) {$(14)_1$};
		\node [style=none] (37) at (0.7, -0.4) {$(41)_k$};
		\node [style=none] (38) at (2.3, -0.4) {$(14)_k$};
		\node [style=none] (39) at (1.5, 8.5) {$\vdots$};
		\node [style=none] (40) at (1.5, 0.75) {$\vdots$};
		\node [style=none] (42) at (-3.6, 4) {$(43)_k$};
		\node [style=none] (43) at (-1.6, 5) {$(34)_1$};
		\node [style=none] (44) at (-1.6, 4) {$(43)_1$};
		\node [style=none] (45) at (-3.6, 5) {$(34)_k$};
    \node [style=none] (46) at (5.3, 4.5) {$\cdots$};
    \node [style=none] (46) at (-2.3, 4.5) {$\cdots$};

		\draw (0) to (2);
		\draw (2) to (3);
		\draw (3) to (1);
		\draw (1) to (0);
		\draw (0) to (16);
		\draw (16) to (17);
		\draw (17) to (2);
		\draw (2) to (18);
		\draw (18) to (19);
		\draw (19) to (0);
		\draw (2) to (20);
		\draw (20) to (5);
		\draw (5) to (3);
		\draw (2) to (6);
		\draw (6) to (7);
		\draw (7) to (3);
		\draw (3) to (8);
		\draw (8) to (9);
		\draw (9) to (1);
		\draw (1) to (11);
		\draw (11) to (10);
		\draw (10) to (3);
		\draw (1) to (12);
		\draw (12) to (13);
		\draw (13) to (0);
		\draw (0) to (15);
		\draw (15) to (14);
		\draw (14) to (1);

\end{tikzpicture}
\end{center}
\caption{Graph $G_k$}
\label{fig:Gk}
\end{figure}
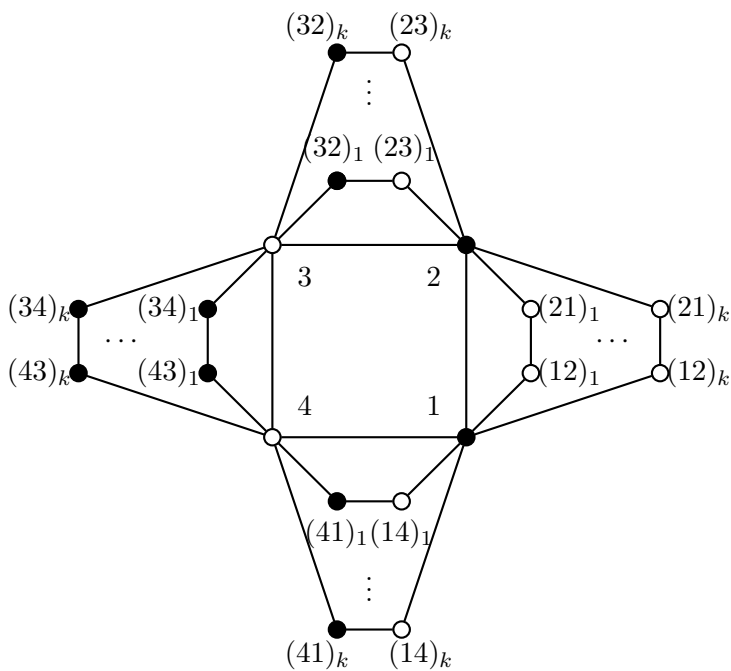

\begin{proposition}
    \label{prop:mu-and-chimu-of-Gk}
If $k\ge 1$, then $\mu(G_k) = 8k$ and $\chimu(G_k) = 2$.
\end{proposition}

\begin{proof}
It is straightforward to check that the set $V(G_k)\setminus [4]$ is a mutual-visibility set, hence $\mu(G_k) \ge 8k$. On the other hand, consider the following four paths: (i) $(14)_1, 1, (12)_1$, (ii) $(21)_1, 2, (23)_1$, (iii) $(32)_1, 3, (34)_1$, and (iv) $(43)_1, 4, (41)_1$. Each of these paths is the unique shortest path between its end vertices, hence an arbitrary mutual-visibility set can contain at most two vertices from each of these four disjoint paths. It follows that $\mu(G_k) \le n(G_k) - 4 = 8k$. We can conclude that $\mu(G_k) = 8k$.

Assume now that
$$X = \{1, 2\} \bigcup \left\{ (32)_i, (41)_i, (34)_i, (43)_i:\ i\in [k]\right\}$$
and let $X' = V(G_k) \setminus X$. See Fig.~\ref{fig:Gk}, where the black vertices are the vertices of $X$. By examining all the possibilities (and taking into account the symmetries), we can verify that $X$ is a mutual-visibility set of $G_k$. Using the symmetry again we get that $X'$ is likewise a mutual-visibility set. Since $V(G_k) = X\cup X'$ we conclude that $\chimu(G_k) = 2$.
\end{proof}

The graphs $G_k$ of Proposition~\ref{prop:mu-and-chimu-of-Gk} thus demonstrate that Algorithm~\ref{alg:greedy} is not optimal in general. Indeed, in view of the proof of the proposition, the algorithm first color all the vertices from the set  $V(G_k)\setminus [4]$ with color $1$. After that the vertices $1, 2, 3, 4$ are yet to be colored. Since they induce a geodetic subgraph of $G_k$ isomorphic to $C_4$, the algorithm colors three of these vertices with color 2 and the remaining vertex with color 3. As $\chimu(G_k) = 2$, the conclusion follows.

Our next goal is to show that Algorithm~\ref{alg:greedy} can return an arbitrarily bad result. For this sake consider the class of graphs $F_k$, $k\ge 2$, obtained from the disjoint union of $C_{4k}$ and $P_{2k-1}$ by adding an edge between a vertex of $C_{4k}$ and one vertex of degree one in $P_{2k-1}$. Let the vertices of $C_{4k}$ be $x_i$, $i\in [4k]$, and let the vertices of $P_{2k-1}$ be $y_i$, $i\in \{2, 3, \ldots, 2k\}$, where the edges in both graphs are natural. Then $F_k$ is obtained from these $C_{4k}$ and $P_{2k-1}$ by adding the edge $x_1y_2$. $F_k$ belongs to the family of the so-called frog graphs, see~\cite{Cicerone-2023}, which also gives rise to the notation for these graphs.

\begin{proposition}
\label{prop:chi-of-frog}
If $k\ge 2$, then $\chimu(F_k) = 2k$.
\end{proposition}

\begin{proof}
Consider the following partition of $V(F_k)$:
$$\{x_1, x_{2k+1}\},  \{y_2, x_2, x_{4k}\},  \{y_3, x_3, x_{4k-1}\}, \ldots, \{y_{2k}, x_{2k}, x_{4k-(2k-2)}\} = \{y_{2k}, x_{2k}, x_{2k+2}\}\,.$$
Since each part of this partition is a mutual-visibility set of $F_k$ we infer that $\chimu(F_k) \le 2k$. On the other hand, we recall from~\cite[Theorem 4.4]{Cicerone-2023} that $\mu(F_k) = 3$. By Lemma~\ref{lem:simple-lower-bound} we then get $$\chimu(G) \ge \left\lceil\frac{n(F_k)}{\mu(F_k)}\right\rceil = \left\lceil\frac{6k - 1}{3}\right\rceil = 2k\,,$$
and we are done.
\end{proof}

Consider now the following sequence of mutual-visibility sets of $F_k$:
$$\{x_1, x_2, x_{2k+1}\},  \{x_3, x_4, x_{2k+3}\}, \ldots, \{x_{2k-1}, x_{2k}, x_{2k + (2k-1)}\} = \{x_{2k-1}, x_{2k}, x_{4k-1}\}\,.$$
Algorithm~\ref{alg:greedy} legally colors the vertices from these respective sets with colors $1, 2\ldots, k$. At that stage of the algorithm, the vertices
$$x_{2k+2}, x_{2k+4}, \ldots, x_{4k}, y_2, y_3, \ldots, y_{2k}$$
are not yet colored. They all lie on a shortest $x_{2k+2},y_{2k}$-path. Moreover, this path is the unique shortest path between $x_{2k+2}$ and $y_{2k}$, therefore a mutual-visibility set can contain at most two vertices out of them. It follows that the algorithm uses at least $\lceil(k + 2k-1)/2 \rceil = \lceil(3k-1)/2 \rceil$ colors for them, hence overall at least $k + \lceil(3k-1)/2 \rceil$ colors are used. This demonstrates that even if an optimal coloring uses color classes that are $\mu$-sets, Algorithm~\ref{alg:greedy} may (by selecting ``wrong" $\mu$-sets) return a coloring using arbitrary more colors than the mutual-visibility chromatic number.

\section{General bounds}
\label{sec:bounds}

Since the mutual-visibility properties of graphs stands on a geodetic distance frame, it is natural to think that for a graph $G$, the parameter $\chimu(G)$ relates to the diameter of $G$. However, such relationships becomes clear only for geodetic graphs. We open this section precisely with the following result on geodetic graphs.

\begin{proposition}\label{prop:diam}
Let $G$ be a geodetic graph. Then
$$\chimu(G) \ge \left\lceil \frac{\diam(G)+1}{2}\right\rceil.$$
Moreover, if $G$ is a block graph, then the equality is achieved.
\end{proposition}

\begin{proof}
Let $P=\{v_i:i\in [\diam(G)+1]\}$ be a diametral path of $G$. Since $G$ is geodetic, no three vertices
in $P$ may have the same color. Therefore, we need at least
$\left\lceil \frac{\diam(G)+1}{2}\right\rceil$ colors in a mutual-visibility coloring of the
vertices of $P$. Hence, we deduce that $\chimu(G) \ge \left\lceil \frac{\diam(G)+1}{2}\right\rceil.$

Assume now that $G$ is a block graph. (Note that block graphs are geodetic because for any two vertices a shortest path between them follows cut vertices along the blocks between them and is hence unique.) The equality holds when $G$ is complete, hence assume in the rest that $\diam(G)\ge 2$. Let $S$ be the set of simplicial vertices of $G$. Then $G-S$ is again a block graph. Let $x,y\in V(G)\setminus S$ such that $d_{G-S}(x,y) = \diam(G-S)$. Hence $x$ and $y$ are cut vertices of $G$ and there exist simplicial vertices $x',y'\in V(G)$, respectively adjacent to $x, y$. Therefore $\diam(G-S) = d_{G-S}(x,y) = d_G(x,y) = d_G(x',y') - 2 = \diam(G)-2$. Since the set $S$ forms a mutual-visibility set, we color its vertices with the same color. Continuing in the same manner on the block graph $G-S$, we arrive to a mutual-visibility coloring using $\left\lceil \frac{\diam(G)+1}{2}\right\rceil$ colors.
\end{proof}

Let $P$ be the Petersen graph. Then it can be checked that $\chimu(P) = 2$, hence the Petersen graph is a sporadic example for which the equality holds in Proposition~\ref{prop:diam}. The latter result also yields the mutual-visibility chromatic number of trees, a result worthy of special mention.

\begin{corollary}
\label{cor:trees}
If $T$ is a tree, then
$$\chimu(T) = \left\{
\begin{array}{ll}
\rad(T)+1;  &  \diam(T)\ {\text even}, \\
\rad(T) ; & \diam(T)\ {\text odd}.
\end{array}
\right.
$$
\end{corollary}

Apart from the trivial bound from Lemma \ref{lem:simple-lower-bound}, our parameter can be related to the classical mutual-visibility number in a different way. The next upper bound appears to be a basic one, but it shows to be very useful. We remark that the bound is somehow derived from the greedy algorithm described in Section~\ref{sec:greedy}.

\begin{proposition}
\label{prop:genbounds}
If $G$ is a connected graph, then
$$\chimu(G) \le \left\lceil\frac{n(G)-\mu(G)+2}{2}\right\rceil.$$
\end{proposition}

\begin{proof}
Let $k=\left\lceil\frac{n(G)-\mu(G)}{2}\right\rceil$ and let $S$ be a $\mu$-set of $G$. Consider $\left\{S_j: \ j \in \left[k\right] \right\}$ as a partition of $V(G)\setminus S$ such that every $S_j$ has cardinality two for every $j \in \left[k\right]-1$, and $|S_{k}|\le 2$.

Now, it can be readily seen that the partition of $V(G)$ given by $\mathcal{P} = S \cup \{S_j: \ j \in \left[k\right] \}$ induces a mutual-visibility coloring of $G$, since $G$ is connected and each $S_j$ has cardinality at most $2$. Thus, $\chimu (G) \le
1+ \left\lceil\frac{n(G)-\mu(G)}{2} \right\rceil = \left\lceil\frac{n(G)-\mu(G)+2}{2} \right\rceil$.
\end{proof}

Since the bound of Proposition~\ref{prop:genbounds} depends on the value of the mutual-visibility number, and computing such parameter is an NP-hard problem (cf. \cite{DiStefano}), the following corollary gives us a tool to bound above the value of $\chimu (G)$ in terms of the maximum degree of a graph. This is based on the fact that for any vertex of a graph $G$, its neighborhood forms a mutual-visibility set of $G$, which implies that $\mu(G)\ge \Delta(G)$.

\begin{corollary}
\label{cor:Deltabounds}
If $G$ is a connected graph, then
    $$\chimu(G) \le \left\lceil\frac{n(G)-\Delta(G)+2}{2}\right\rceil.$$
\end{corollary}

Let us point out that the condition for $G$ to be connected is essential for applying this result.
For example, consider the infinite family of graphs $G(p,q)=K_p \cup \overline{K_q}$. In such case, for every $q\ge 3$ it holds that
\begin{align*}
\chimu(G(p,q)) & = q+1 > \left\lceil\frac{n(G(p,q))-\Delta(G(p,q))+2}{2} \right\rceil \\
   & =\left\lceil\frac{p+q-(p-1)+2}{2} \right\rceil
   =\left\lceil\frac{q+3}{2} \right\rceil\,.
\end{align*}

For graphs containing a vertex adjacent to all other vertices,  the following consequence is deduced.

\begin{corollary}
If $G$ is not complete and $\Delta(G)=n(G)-1$, then $\chimu(G)=2$.
\end{corollary}

\begin{proof}
Since $\Delta(G)=n-1$ and $G$ is not complete, $\mu(G) = n-1$. Therefore, Lemma~\ref{lem:simple-lower-bound} and Corollary~\ref{cor:Deltabounds} became equalities and we may derive that $\chimu(G)=2.$
\end{proof}

Observe that the maximum possible value of the upper bound from Corollary~\ref{cor:Deltabounds} is $\left\lceil\frac{n(G)}{2}\right\rceil$ which appears when $\Delta(G)=2$. We next characterize the graphs achieving precisely this maximum value.

\begin{proposition}\label{prop:oddpath}
If $G$ is a graph, then $\frac{n(G)}{2}<\chimu(G)=\left\lceil\frac{n(G)}{2}\right\rceil$
if and only if $G$ is a path of odd order.
\end{proposition}

\begin{proof}
Let $n = n(G)$ and $\Delta = \Delta(G)$. First, let us suppose that $G=P_{2k+1}.$ By Lemma~\ref{lem:simple-lower-bound} and Proposition~\ref{prop:genbounds} we have that $\chimu(G)=\left\lceil \frac{n}{2} \right\rceil.$ On the other hand, assume now that $\frac{n}{2}< \chimu(G)=\left\lceil \frac{n}{2} \right\rceil$ which implies that $n$ is an odd integer. By Proposition~\ref{prop:genbounds} we deduce that
$$ \left\lceil \frac{n}{2} \right\rceil= \left\lceil\frac{n-\Delta+2}{2} \right\rceil. $$
Since $n$ is odd, we have that $ \frac{n+1}{2} = \frac{n+1}{2}+ \left\lceil\frac{1-\Delta}{2} \right\rceil $, and therefore, $\Delta \le 2.$ Clearly, $\Delta=1$ implies that $G=K_2$, because $G$ is connected, which is not possible because $n$ is odd. Hence, it is derived that $\Delta =2$ and $G$ must be either an odd cycle, or a path of odd order.

If $G$ is a cycle, then it is known that $\mu(G)=3.$ Let $S$ be a mutual visibility set of order $3$ in the cycle. As $|V(G)\setminus S|=n-3$ is even, we may consider a partition of $V(G)\setminus S$ into $(|V(G)\setminus S|)/2$ sets of cardinality $2$, which are mutual-visibility sets. Then
$$ \chimu(G) \le 1+ \frac{|V(G)\setminus S|}{2}=1+\frac{n-3}{2}=\frac{n-1}{2}<\frac{n+1}{2}=\chimu(G),$$
which is a contradiction. Therefore, we conclude that $G$ must be a path with an odd number of vertices.
\end{proof}

The following realization result complements Proposition~\ref{prop:oddpath}.

\begin{proposition}
    Let $n,k$ be positive integers such that
    $2 \le  k \le \left\lfloor\frac{n}{2}\right\rfloor$. Then
    there is a graph $G$ with $n = n(G)$ and
    $$\left\lceil\frac{n(G)}{\mu(G)}\right\rceil \le \chimu(G)=k \le \left\lfloor\frac{n(G)}{2}\right\rfloor.$$
\end{proposition}

\begin{proof}
If $k=\left\lfloor\frac{n}{2}\right\rfloor$, then it is sufficient to consider the path graph $G=P_{2k}$.
From now on, we may assume that $2 \le  k \le \left\lfloor\frac{n}{2}\right\rfloor -1.$
Let $G$ be the graph obtained by joining a leaf of a path having $p=2k-1$ vertices with $q=n-2k+1$ isolated vertices, see Fig.~\ref{broom}.

\begin{figure}[ht!]
\centering
\begin{tikzpicture}[scale=.5, transform shape, style=thick]
\tikzstyle{none}=[scale=0.7,fill=none, draw=none, shape=circle]
\tikzstyle{black_V}=[fill=black, draw=black, shape=circle]

\tikzstyle{dashed_E}=[-, dashed]

		\node [style={black_V}] (0) at (7, 0) {};
		\node [style={black_V}] (1) at (5, 0) {};
		\node [style={black_V}] (2) at (3, 0) {};
		\node [style={black_V}] (3) at (1, 0) {};
		\node [style={black_V}] (4) at (-4, 0) {};
		\node [style={black_V}] (5) at (-6, 0) {};
		\node [style={black_V}] (6) at (-8, 0) {};
		\node [style={black_V}] (7) at (12, 5) {};
		\node [style={black_V}] (8) at (12, 3) {};
		\node [style={black_V}] (9) at (12, 1) {};
		\node [style={black_V}] (10) at (12, -1) {};
		\node [style={black_V}] (11) at (12, -3) {};
		\node [style={black_V}] (12) at (12, -8) {};
		\node [style={black_V}] (13) at (12, 7) {};
		\node [style=none] (14) at (12, -3.5) {};
		\node [style=none] (15) at (12, -7) {};
		\node [style=none] (16) at (-3, 0) {};
		\node [style=none] (17) at (0, 0) {};
		\node [style=none] (18) at (-8, 1) {\Huge $v_1$};
		\node [style=none] (19) at (-6, 1) {\Huge $v_2$};
		\node [style=none] (20) at (-4, 1) {\Huge $v_3$};
		\node [style=none] (21) at (4.75, 1) {\Huge $v_{p-1}$};
		\node [style=none] (22) at (6.5, 1) {\Huge $v_{p}$};
		\node [style=none] (23) at (13.5, 7) {\Huge $v_{p+1}$};
		\node [style=none] (24) at (13.5, 5) {\Huge $v_{p+2}$};
		\node [style=none] (26) at (13.5, -8) {\Huge $v_{p+q}$};

            \draw (0) to (12);
		\draw (3) to (0);
		\draw (0) to (8);
		\draw (7) to (0);
		\draw (0) to (9);
		\draw (13) to (0);
		\draw (10) to (0);
		\draw (11) to (0);
		\draw [style={dashed_E}] (14.center) to (15.center);
		\draw [style={dashed_E}] (6) to (4);
		\draw [style={dashed_E}] (17.center) to (16.center);
	
\end{tikzpicture}
\caption{A graph $G$ or order $n=p+q$ with $\chimu(G)=k=\frac{p+1}{2}$.} \label{broom}
\end{figure}
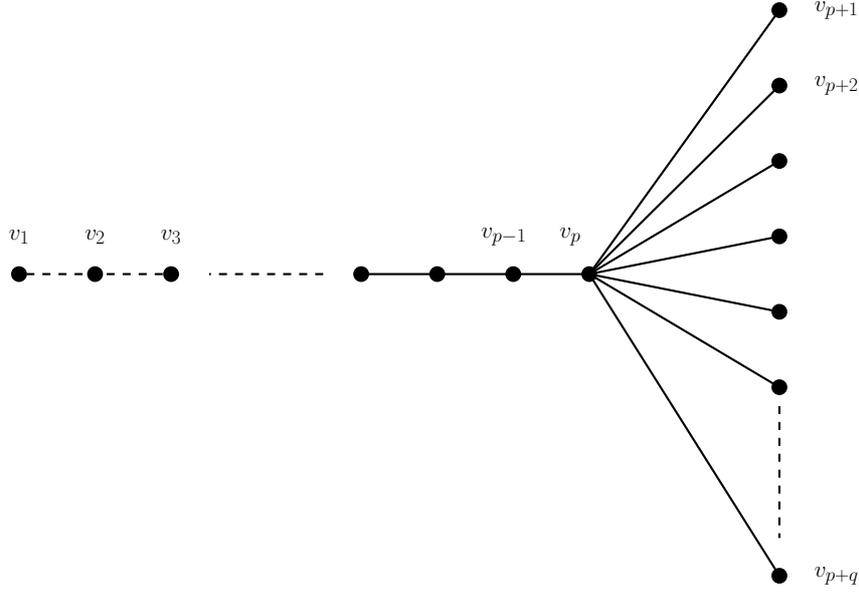

Since $4 \le 2k \le n-2$, the graph $G$ is well-defined and we have that $3\le p,q\le n-3.$
It is straightforward to check that the set of vertices $S=\{v_1,v_{p+1},\ldots , v_{p+q}\}$ is
a mutual-visibility set of maximum cardinality $q+1$. Hence, $\mu(G)=q+1.$

Note that $G$ is a geodetic graph with $\diam(G)=2k-1$. Then, by Proposition \ref{prop:diam}, we have that
$\chimu(G) \ge k.$ On the other hand, let ${\cal P}=\{ P_i= \{ v_{2i},v_{2i+1} \}:i\in [k-1] \}\cup \{ S \}$
is a mutual-visibility coloring of cardinality $k$. So, $\chimu(G) \le k$ and the equality holds.

Finally, we have to prove that $\left\lceil \frac{n(G)}{\mu(G)}\right\rceil \le k=\chimu(G)$. Note that
$n=p+q$, $p=2k-1$, and $\mu(G)=q+1$. Therefore, $\left\lceil\frac{n(G)}{\mu(G)}\right\rceil \le k$ if and
only if $\frac{p+q}{q+1} \le \frac{p+1}{2}$, which is equivalent to $p+q-1\le pq$.
Since $p,q> 1$, it is easy to check that the last inequality holds because $(p-1)(q-1)\ge 0$.
\end{proof}

Next, we show an upper bound for regular graphs that improves the one given by Proposition~\ref{prop:genbounds}.

\begin{proposition}\label{prop:regular}
If $G$ is an $r$-regular graph, $r\ge 2$, with girth $g(G)\ge 6$, then
    $$ \chimu(G) \le \left\lceil \frac{n(G)-r^2+4}{2} \right\rceil\,.$$
\end{proposition}

\begin{proof}
Set $n = n(G)$. Let $v\in V(G)$ and set $S=N_G(v)$ and $T=N_G(N_G(v))-v$. Since $g(G)\ge 6$, we infer that $S\cap T = \emptyset$, that $G[T]$ is an edgeless graph, and $|T|=r(r-1)$. And, clearly, $|S|=r$.

The set $S$  is a mutual-visibility set since it is an open neighborhood of a vertex. We further claim that $T$ is a mutual-visibility set as well. Consider arbitrary vertices $x, y\in T$ and note that $d_G(x,y)\le 4$. We wish to see that $x$ and $y$ are $T$-visible. Since $G[T]$ is an edgeless graph, there is nothing to show if $d_G(x,y)\le 2$. The fact that $G[T]$ is an edgeless graph also implies that if $d_G(x,y) = 3$, then no shortest $x,y$-path contains a vertex of $T$ as an interior vertex. Finally, if $d_G(x,y) = 4$, then there exits a shortest $x,y$-path containing $v$, hence also in this case $x$ and $y$ are $T$-visible. We can conclude that $T$ is a mutual-visibility set.

    The set of vertices $V(G)\setminus \{S,T\}$ has cardinality $n-r^2$, let its vertices be $v_i$, $i\in[n-r^2]$. Consider now the partition of $V(G)\setminus \{S,T\}$ given by
    $\mathcal{P}=\{ P_i=\{v_{2i-1},v_{2i}\}  :  i\in [\frac{n-r^2}{2}] \}$, if $n-r^2$ is even or
    $\mathcal{P}=\{ P_i=\{v_{2i-1},v_{2i}\}  :  i\in [\frac{n-r^2-1}{2}] \} \cup
    \{v_{n-r^2}\}$,  if $n-r^2$ is odd. Since every single element of $\mathcal{P}$
    has cardinality at most $2$, they are mutual-visibility sets of $G$. So, we can color $S$ and
    $T$ with two different colors and we may assign one more distinct color to each set $P_i$. Hence,
    $$ \chimu(G) \le 2+ \left\lceil\frac{n-r^2}{2}\right\rceil=\left\lceil\frac{n-r^2+4}{2}\right\rceil$$
    which we needed to prove.
\end{proof}

We can note that the bound given by Proposition~\ref{prop:regular} improves the one described in
the Proposition~\ref{prop:genbounds} for regular graphs, because $r^2-2\ge r=\Delta(G)$ for every
$r$-regular graph $G$.

We end this section with a Nordhaus-Gaddum-type upper bound related to the mutual-visibility chromatic number of a graph.

\begin{proposition}\label{prop:nordhaus}
If $G$ is a connected graph such that $\overline{G}$ is also connected, then
    \begin{equation*}\label{nordh}
     \chimu(G)+\chimu(\overline{G}) \le \left\lceil \frac{n(G)-\mu(G)+2}{2} \right\rceil+
    \left\lceil \frac{\delta(G)+3}{2} \right\rceil\,.
    \end{equation*}
\end{proposition}

\begin{proof}
Since $G$ is connected, Proposition~\ref{prop:genbounds} yields    $\chimu(G)\le \left\lceil \frac{n(G)-\mu(G)+2}{2} \right\rceil.$ Also, as $\overline{G}$ is connected, Corollary~\ref{cor:Deltabounds} gives  $\chimu(\overline{G})\le \left\lceil \frac{n(G)-\Delta(\overline{G})+2}{2} \right\rceil=\left\lceil \frac{\delta(G)+3}{2} \right\rceil.$ Therefore,
$$ \chimu(G)+\chimu(\overline{G}) \le \left\lceil \frac{n(G)-\mu(G)+2}{2} \right\rceil+\left\lceil \frac{\delta(G)+3}{2} \right\rceil$$
which is the desired bound.
\end{proof}

It can be readily seen that the bound given by Proposition~\ref{prop:nordhaus} is tight for paths
of odd order $G=P_{2k+1}$. Since $P_{2k+1}$ is a bipartite graph, it holds that $\chimu(\overline{G})=2$. By Proposition~\ref{prop:oddpath}, we know that $\chimu(P_{2k+1})=k+1.$ Therefore equality holds in Proposition~\ref{prop:nordhaus}.

In addition to the comments above, Proposition \ref{prop:nordhaus} also leads to the following conclusion by taking into account that $\chimu(G)\ge \Delta(G)$ for any graph $G$.

\begin{corollary}\label{coro:nordhaus}
If $G$ is a connected graph such that $\overline{G}$ is also connected, then
    \begin{equation*}
     \chimu(G)+\chimu(\overline{G}) \le \left\lceil \frac{n(G)+5}{2} \right\rceil.
    \end{equation*}
\end{corollary}

Notice that again the case of paths of odd order $G=P_{2k+1}$ can be used to show the tightness of the upper bound of Corollary \ref{coro:nordhaus}.

\section{Corona product graphs}
\label{sec:corona}

Let $G$ be a graph whose vertex set is $V(G) = \{v_1, \ldots ,v_{n}\}$, and let $H$ be a graph. The {\em corona product} $G\odot H$ of $G$ and $G$ is the graph defined as follows. We take one copy of $G$ and $n(G)$ disjoint copies of $H$, denoted $H^1,\ldots, H^{n}$. Next, for every integer $i\in [n]$, we add all the possible edges between any vertex $v_i\in V(G)$ and all the vertices of $H^i$.

\begin{proposition}
\label{prop:corona}
If $G$ is a connected graph of order at least two and $H$ is a graph, then
$$ \chimu(G) \le \chimu(G\odot H) \le \chimu(G)+1.  $$
\end{proposition}

\begin{proof}
To prove the lower bound, let us consider a mutual-visibility coloring of $G\odot H$ given by
the collection of sets $\mathcal{P}=\{P_i\},$ with $i\in [\chimu(G\odot H)].$ Since $G$ is a convex subgraph of $G\odot H$, it holds that the collection $\mathcal{P'}=\{P_i\cap V(G)\}$ is a mutual-visibility coloring of $G$. Therefore $\chimu(G)\le\chimu(G\odot H)$.

Next, we prove the upper bound. If $G$ is a complete graph $K_n$, then it can be readily observed that $\chimu(K_n\odot H) = 2 = \chimu(K_n)+1$, since the vertex set $V(K_n)$ and the complement of it in $K_n\odot H$ are both mutual-visibility sets of $K_n\odot H$. Moreover, if $G$ is the graph $P_3$, then $G\odot H$ contains a geodetic subgraph isomorphic to $P_5$, hence it can be deduced that $\chimu(P_3\odot H) = 3 = \chimu(P_3)+1$. Thus, from now on we may assume that $G$ is a non-complete graph of order at least $4$.

Let $\mathcal{P}=\{P_1,\dots,P_r\}$ be a mutual-visibility coloring of $G$. We claim that $\mathcal{P}\cup W$, where $W=\bigcup_{i\in [n]}V(H^i)$, is a mutual-visibility coloring of $G\odot H$.

Clearly, any set $P_i\in \mathcal{P}$ is also a mutual-visibility set of $G\odot H$. On the other hand, if $x,y\in V(H^i)$ for some $i\in [n]$, then they are either adjacent or at distance two. Since they have a common neighbor (the vertex $v_i$) which is not in $W$, it holds that $x,y$ are $W$-visible when they are not adjacent. Assume now that $x\in V(H^i)$ and $y\in V(H^j)$ for some distinct $i,j\in [n]$. Hence, they are clearly $W$-visible by using any shortest $v_i,v_j$-path. As a consequence, it follows that $\mathcal{P}\cup W$ is a mutual-visibility coloring of $G\odot H$ as claimed, and so, $\chimu(G\odot H)\le \chimu(G)+1$.
\end{proof}

It is worth noting that the bounds given by Proposition~\ref{prop:corona} are tight. To see this, let us consider the
cycle $C_4$ with vertex set $V(C_4)=[4]$. Let us denote by $K_m^i$ the copy of $K_m$ corresponding to the vertex $i$ in $C_4\odot K_m$. We can readily check that $\chimu(C_4\odot K_m)=\chimu(C_4)=2$,
by considering the mutual-visibility coloring of $C_4\odot K_m$ given by the sets of vertices
$\{\{1,2, V(K_m^3), V(K_m^4)\}, \{3,4, V(K_m^1), V(K_m^2)\}\}$ (see Fig.~\ref{corona_lower}).
Similarly, we can verify that $\chimu(K_n\odot K_m)=\chimu(K_n)+1=2$ for any integers $m,n$ with $n\ge 2$ and $m\ge 1$.

\begin{figure}[ht!]
    \centering
\begin{tikzpicture}[scale=0.85, transform shape, style=thick]
            \node [style={black vertex}] (0) at (-1.5, 1.5) {};
		\node [style={black vertex}] (1) at (-1.5, -1.5) {};
		\node [style={white vertex}] (2) at (1.5, 1.5) {};
		\node [style={white vertex}] (3) at (1.5, -1.5) {};
		\node [style=none] (8) at (-2, 1) {$1$};
		\node [style=none] (9) at (2, 1) {$3$};
		\node [style=none] (10) at (2, -1) {$4$};
		\node [style=none] (11) at (-2, -1) {$2$};
		\node [style={white vertex}] (12) at (-3, 5.65) {};
		\node [style={white vertex}] (13) at (-4.25, 4.5) {};
		\node [style={white vertex}] (14) at (-3.5, 3) {};
		\node [style={white vertex}] (15) at (-2.5, 3) {};
		\node [style={white vertex}] (16) at (-1.75, 4.5) {};
		\node [style={white vertex}] (17) at (-3, -2.5) {};
		\node [style={white vertex}] (18) at (-4.25, -3.65) {};
		\node [style={white vertex}] (19) at (-3.5, -5.15) {};
		\node [style={white vertex}] (20) at (-2.5, -5.15) {};
		\node [style={white vertex}] (21) at (-1.75, -3.65) {};
		\node [style={black vertex}] (22) at (3, 5.65) {};
		\node [style={black vertex}] (23) at (1.75, 4.5) {};
		\node [style={black vertex}] (24) at (4.25, 4.5) {};
		\node [style={black vertex}] (25) at (2.5, 3) {};
		\node [style={black vertex}] (26) at (3.5, 3) {};
		\node [style={black vertex}] (27) at (3, -2.5) {};
		\node [style={black vertex}] (28) at (1.75, -3.65) {};
		\node [style={black vertex}] (29) at (4.25, -3.65) {};
		\node [style={black vertex}] (30) at (2.5, -5.15) {};
		\node [style={black vertex}] (31) at (3.5, -5.15) {};
		\draw (0) to (2);
		\draw (2) to (3);
		\draw (3) to (1);
		\draw (1) to (0);
		\draw (14) to (16);
            \draw (13) to (16);
            \draw (12) to (14);
		\draw (14) to (15);
		\draw (15) to (16);
		\draw (16) to (12);
		\draw (12) to (15);
		\draw (15) to (13);
		\draw (13) to (12);
		\draw (13) to (14);

		\draw (17) to (19);
		\draw (19) to (20);
		\draw (20) to (21);
		\draw (21) to (17);
		\draw (17) to (20);
		\draw (21) to (18);
		\draw (18) to (17);
		\draw (18) to (19);
            \draw (18) to (20);
            \draw (19) to (21);
		
		\draw (22) to (25);
		\draw (25) to (24);
		\draw (24) to (23);
		\draw (23) to (26);
		\draw (26) to (22);
		\draw (22) to (24);
		\draw (24) to (26);
		\draw (26) to (25);
		\draw (25) to (23);
		\draw (23) to (22);
		
		\draw (27) to (30);
		\draw (30) to (29);
		\draw (29) to (28);
		\draw (28) to (31);
		\draw (31) to (27);
		\draw (27) to (29);
		\draw (29) to (31);
		\draw (31) to (30);
		\draw (30) to (28);
		\draw (28) to (27);
		
            \draw (1) to (17);
		\draw (1) to (21);
		\draw (1) to (18);
		\draw (1) to (19);
		\draw (1) to (20);
		\draw (0) to (14);
		\draw (0) to (15);
		\draw (0) to (16);
		\draw (0) to (12);
		\draw (0) to (13);
            \draw (2) to (23);
		\draw (2) to (22);
		\draw (2) to (25);
		\draw (2) to (26);
		\draw (2) to (24);
            \draw (3) to (27);
		\draw (3) to (28);
		\draw (3) to (29);
		\draw (3) to (31);
		\draw (3) to (30);
\end{tikzpicture}
\caption{A mutual-visibility coloring of $C_4 \odot K_5$ with $\chimu(C_4) = 2$ colors.}
\label{corona_lower}
\end{figure}
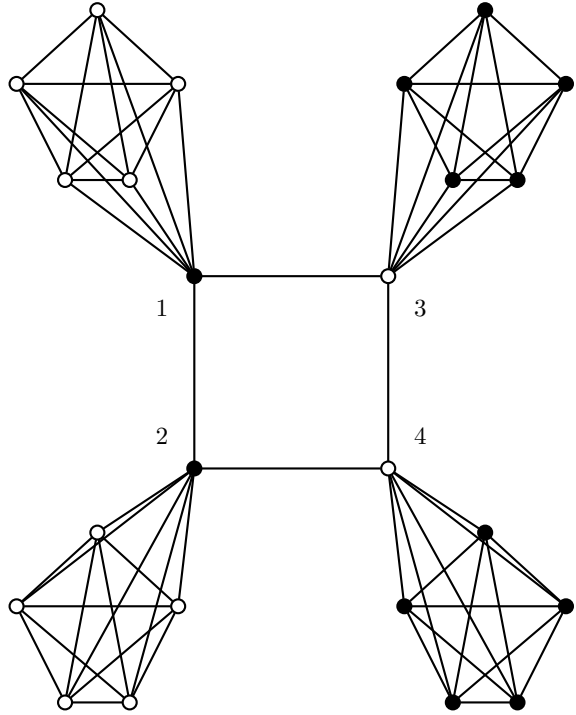

\section{Graphs with mutual-visibility chromatic number two}
\label{sec:chimu2}

It is straightforward to observe that $\chimu(G)=1$ if and only if $G$ is a complete graph. In this sense, the smallest non trivial value for $\chimu(G)$ is precisely $2$. This motivates the study of this section.

Consider the family $\mathcal{A}$ of graphs $G_A$ defined as follows. We begin with two graphs $A_r$ and $A_s$ and two complete graphs $K_r$ and $K_s$. Next, we add all possible edges between the vertices of $A_r$ and the vertices of $K_r$, as well as, all the possible edges between the vertices of $A_s$ and the vertices of $K_s$. Then, to obtain a graph $G_A\in \mathcal{A}$, we add some edges between the vertices of $A_r$ and $A_s$ in such a way that any two not adjacent vertices of $A_r$ will have a common neighbor in $V(A_s)$, and that any two not adjacent vertices of $A_s$ will have a common neighbor in $V(A_r)$.

We also consider the family $\mathcal{B}$ of graphs $G_B$ constructed as follows. We begin with two graphs $B$ and $B'$ satisfying that $|\diam(B)-\diam(B')|\le 2$. Next, for any two not adjacent vertices $x,y\in V(B)$ with $d_B(x,y)=d$, we  select vertices $x',y'\in V(B')$ such that $d_{B'}(x',y')=d-2$, and add the edges $xx'$ and $yy'$. (Note that if $d=2$, then $x' = y'$.) Similarly, for any two not adjacent vertices $x',y'\in V(B')$ with $d_{B'}(x',y')=d$, we  select two vertices $x,y\in V(B)$ such that $d_B(x,y)=d-2$, and add the edges $xx'$ and $yy'$.

By inspection we can check that the graphs $G_k$ of Proposition~\ref{prop:mu-and-chimu-of-Gk} belong to the family $\mathcal{B}$. More generally, we have the following result.

\begin{proposition}
\label{prop:family_F}
If $G\in \mathcal{A}\cup \mathcal{B}$, then $\chimu(G)=2$.
\end{proposition}

\begin{proof}
Assume $G=G_A\in \mathcal{A}$. According to the construction of $G_A$, we can readily see that any two vertices of the set $S=V(A_r)\cup V(K_r)$ are $S$-visible, since they are either adjacent or at distance two (in such case $x,y\in V(A_r)$). In the latter case, by the construction, any two not adjacent vertices of $V(A_r)$ have a common neighbor  $z\in V(A_s)$ and $z\notin S$. Thus, $S$ is a mutual-visibility set of $G$. Similarly, the set $S'=V(B_s)\cup V(K_s)$ is also a mutual-visibility set of $G$. These facts allow to conclude that $\chimu(G_A)=2$.

Assume next that $G=G_B\in \mathcal{B}$. In such situation, we claim that each set $V(B)$ and $V(B')$ are mutual-visibility sets of $G_B$. Notice that for any two not adjacent vertices $x,y\in V(B)$ with $d_B(x,y)=d\ge 2$ there is a shortest $x,y$-path whose internal vertices are in $V(B')$ according to the construction of $G_B$. Thus, $x,y$ are $V(B)$-visible and so, $V(B)$ is a mutual-visibility set of $G_B$ as claimed. The same conclusion can be readily deduced for $V(B')$. Thus, $\chimu(G_B)=2$.
\end{proof}

Another infinite family of graphs $G$ satisfying $\chimu(G)=2$ is  given by the strong product graphs $H\boxtimes K_2$ for any non-complete graph $H$. From the proof of~\cite[Theorem 5.1]{DiStefano} one can deduce that if $V(K_2)=\{u,v\}$, then the two sets $V(H)\times \{u\}$ and  $V(H)\times \{v\}$ are mutual-visibility set of $H\boxtimes K_2$, and so, $\chimu(H\boxtimes K_2)\le 2$. The equality follows because $H$ is not a complete graph. Yet another family of graphs $G$ satisfying that $\chimu(G)=2$ is obtained from the last comments of Section \ref{sec:corona}. That is, for any corona graph $K_n\odot K_m$ ($n\ge 2$ and $m\ge 1$) it holds that $\chimu(K_n\odot K_m)=2$.

The comments of the paragraph above and the results from Proposition~\ref{prop:family_F} allow to observe that there are graphs $G$ satisfying that  $\chimu(G)=2$ and whose structures are intrinsically different. Note that the graphs $G_A\in \mathcal{A}$ satisfy that $\diam(G_A)\le 3$, while the graphs $G_B\in \mathcal{B}$ might have diameter as large as we would require.

\section{Concluding remarks}
\label{sec:conclud}

In this section we propose some open problems and suggestions that might be taken into account for further investigation of the concept introduced in this paper. First, the following is a fundamental question to be investigated.

\begin{problem}
Investigate the computational complexity of computing the mutual-visibility chromatic number.
\end{problem}

In Section \ref{sec:diameter2}, the value of the $\chimu(G)$ is computed (asymptotically) for the case of 2-dimensional Hamming graphs. It is already known that finding the mutual-visibility number of Hamming graphs is a very challenging problem (see for instance \cite{Cicerone-2023}). In this sense, the following problem might be also challenging, but worth of exploring.

\begin{problem}
Compute or at least bound the mutual-visibility chromatic number of Hamming graphs, and of the Cartesian product of at least two complete graphs in general.
\end{problem}

In view of Proposition~\ref{pro:chimu-chi}, and due to the fact that the chromatic number of graphs is one of the most classical parameters in graph theory, we pose the following problem.

\begin{problem}
Characterize the graphs $G$ of diameter two for which $\chimu(G) = \chi(G)$ holds. In particular, are there any additional such graphs besides the complete bipartite graphs? In addition, are there some other relationships between $\chimu(G)$ and $\chi(G)$ for other graph classes?
\end{problem}

Axenovich and Liu~\cite{Axenovich-2024} proved that $\mu(Q_n) > 0.186 \cdot 2^n$ for any hypercube graph $Q_n$. This raises the following:

\begin{question}
Does there exist an absolute constant $M$ such that $\chimu(Q_n) < M$ holds for any integer $n$?
\end{question}

Section \ref{sec:chimu2} deals the graphs $G$ satisfying that $\chimu(G) = 2$ and Corollary \ref{coro:nordhaus} shows a Nordhauss-Gaddum type result for $\chimu(G)$. In this sense, it seems to be worth of considering the following.

\begin{problem}
Characterize the graphs $G$ (at least partially) satisfying that $\chimu(G) = 2$, as well as, those connected graphs $G$ whose complements are also connected, and for which $\chimu(G)+\chimu(\overline{G}) = \left\lceil \frac{n(G)+5}{2} \right\rceil$.
\end{problem}

A variety of mutual-visibility problems in graph theory (variants called total, outer and dual) was described in \cite{variety-2023}. In this sense, it seems to be natural to consider the problem of coloring the vertices of a graph with total, outer or dual mutual-visibility sets. Analogously, closely related to the mutual-visibility sets of graphs there exist the general position sets, see~\cite{irsic-2024, KorzeVesel, Manuel-2018, tian-2023, yao-2022}. Hence, coloring the vertices of a graph with general position sets is also of interest.

\section*{Acknowledgements}

We may remark that this investigation was in part motivated by an informal discussion with Jarek Grytczuk (Cracow, Poland) at the $8^{\rm th}$ Gda\'nsk Workshop on Graph Theory, 2023, where I.\ G.\ Yero presented a talk on some mutual-visibility problems in graphs.

S.\ Klav\v{z}ar has been supported by the Slovenian Research Agency ARIS (research core funding P1-0297 and projects N1-0285, N1-0218). D.\ Kuziak and I.\ G.\ Yero have been partially supported by the Spanish Ministry of Science and Innovation through the grant PID2023-146643NB-I00. J.\ C.\ Valenzuela has been partially supported by the Spanish Ministry of Science and Innovation through the grant PID2022-139543OB-C41. Moreover, this investigation was completed while D.\ Kuziak was visiting the University of Ljubljana supported by “Ministerio de Educaci\'on, Cultura y Deporte”, Spain, under the “Jos\'e Castillejo” program for young researchers (reference number: CAS22/00081).

\section*{Author contributions statement}
All authors contributed equally to this work.

\section*{Conflicts of interest}
The authors declare no conflict of interest.

\section*{Data availability}
No data was used in this investigation.

\end{document}